\pdfoutput=1
\documentclass[a4paper,11pt]{article}
\usepackage[english]{babel}
\usepackage[sc,osf]{mathpazo}
\usepackage[T1]{fontenc}
\usepackage[utf8]{inputenc}
\usepackage{amsfonts}
\usepackage{amssymb}
\usepackage{eucal}
\usepackage{mathrsfs}
\usepackage{amsmath}
\usepackage{amsthm}
\usepackage{mathtools}
\usepackage{stmaryrd}
\usepackage{thmtools}
\usepackage{graphicx}
\usepackage[dvipsnames]{xcolor}
\usepackage[shortlabels]{enumitem}
\usepackage[pdfborder={0 0 0}]{hyperref}
\usepackage[nameinlink,capitalize]{cleveref}
\usepackage{tikz-cd}
\usepackage{csquotes}
\usepackage{microtype}
\usepackage[backend=biber,style=alphabetic,sorting=nty,giveninits,maxalphanames=4]{biblatex}
\renewbibmacro{in:}{
	\ifentrytype{article}{}{\printtext{\bibstring{in}\intitlepunct}}}
\numberwithin{equation}{section}
\makeatletter
\renewcommand*{\eqref}[1]{%
  \hyperref[{#1}]{\textup{\tagform@{\ref*{#1}}}}%
}
\makeatother

\tikzset{%
    symbol/.style={%
        draw=none,
        every to/.append style={%
            edge node={node [sloped, allow upside down, auto=false]{$#1$}}}
    }
}
\tikzset{
    rot90/.style={anchor=south, rotate=90, inner sep=.5mm}
}
\tikzset{
    rot320/.style={anchor=south, rotate=320, inner sep=.5mm}
}

\addto\extrasenglish{ 
 
 }

\declaretheorem[name=Theorem, numberwithin=section]{theorem}
\declaretheorem[name=Theorem, numbered=no]{theorem*}
\declaretheorem[name=Theorem]{theoremA}

\declaretheorem[name=Lemma, sibling=theorem]{lemma}
\declaretheorem[name=Proposition, sibling=theorem]{proposition}

\declaretheorem[style=definition, name=Definition, sibling=theorem]{definition}
\declaretheorem[style=definition, name=Remark, sibling=theorem]{remark}
\declaretheorem[style=definition, name=Example, sibling=theorem]{example}

\DeclareMathOperator{\Fun}{Fun}

\DeclareMathOperator*{\colim}{colim}
\DeclareMathOperator{\id}{id}

\DeclareMathOperator{\Ind}{Ind}
\DeclareMathOperator{\MapSF}{\Map_{\sSF}}
\DeclareMathOperator{\MapTop}{\Map_{\Top}}
\DeclareMathOperator{\MapTopp}{\Map_{\Top_*}}
\DeclareMathOperator{\MapsSet}{\Map}
\DeclareMathOperator{\MapsSetp}{\Map_*}
\DeclareMathOperator{\MapSpW}{\Map_{\SpW}}
\DeclareMathOperator{\Indinfty}{\Ind_{\infty}}

\DeclareMathOperator{\Ob}{Ob}

\DeclareMathOperator{\Map}{Map}

\DeclareMathOperator{\Sp}{Sp}

\DeclareMathOperator{\SingTop}{Sing}

\DeclareMathOperator{\hcN}{N_{\mathrm{hc}}}

\newcommand{\op}[1]{\overline{#1}}

\newcommand{\bbN}{\mathbb N}
\newcommand{\bbR}{\mathbb R}

\newcommand{\bbC}{\mathbb C}

\newcommand{\bfD}{\mathbf D}

\newcommand{\bfT}{\mathbf T}

\newcommand{\bbU}{\mathbb U}
\newcommand{\bbT}{\mathbb T}
\newcommand{\bbPT}{\mathbb{PT}}
\newcommand{\bbSU}{\mathbb{SU}}

\newcommand{\s}{\mathbf s}

\newcommand{\Set}{\mathbf{Set}}
\newcommand{\Top}{\mathbf{Top}}

\newcommand{\CAlg}{\mathbf{C}^*}
\newcommand{\cCAlg}{\mathbf{c}\CAlg}
\newcommand{\CHaus}{\mathbf{CH}}
\newcommand{\MAlgu}{\mathbf{M}}
\newcommand{\NCWf}{\mathbf{NCW}_f}
\newcommand{\CWf}{\mathbf{CW}_f}

\newcommand{\pfinsSet}{\s\Set_*^{\mathrm{fin}}}
\newcommand{\calM}{\mathcal M}
\newcommand{\scalM}{\mathcal M_\Delta}
\newcommand{\tcalM}{\mathcal M_\Top}

\newcommand{\SpW}{\mathbf{Sp}^\mathscr{W}}
\newcommand{\sSF}{\mathbf{SF}}
\newcommand{\sSFlin}{\mathbf{SF}_l}
\newcommand{\SF}{\mathrm{SF}}
\newcommand{\SFlin}{\SF_l}

\makeatletter
\newcommand{\oset}[3][0ex]{%
  \mathrel{\mathop{#3}\limits^{
    \vbox to#1{\kern-2\ex@
    \hbox{$\scriptstyle#2$}\vss}}}}
\makeatother

\newcommand{\cofarrow}{\rightarrowtail}

\newcommand{\wearrow}{\oset[-.16ex]{\sim}{\longrightarrow}}
\newcommand{\fibarrow}{\twoheadrightarrow}

\renewcommand{\subset}{\subseteq}

\mathchardef\mhyphen="2D

\title{A note on noncommutative CW-spectra}
\author{Thomas Blom}
\date{\today}

\begin{document}

\maketitle
\begin{abstract}
    We use the machinery of \cite{BlomMoerdijk2020SimplicialProV1} to give an alternative proof of one of the main results of \cite{AroneBarneaSchlank2021SpectralPresheavesV2}. This result states that the category of noncommutative CW-spectra can be modeled as the category of spectral presheaves on a certain category $\calM$, whose objects can be thought of as ``suspension spectra of matrix algebras''. The advantage of our proof is that it mainly relies on well-known results on (stable) model categories.
\end{abstract}

\section{Introduction}

The well-known Gelfand duality theorem states that the category of pointed compact Hausdorff spaces is dual to the category of commutative $C^*$-algebras. This result motivates the philosophy that general $C^*$-algebras are dual to ``noncommutative spaces''. A natural question arising from this philosophy is what the homotopy theory of $C^*$-algebras (or noncommutative spaces) should look like, and a natural follow-up question is then to describe their stable homotopy theory. There have been many approaches to answering these questions, see e.g.\ \cite{Thom2003ConnectiveBivariant}, \cite{Ostvaer2010HomotopyCstar}, \cite{Uuye2013Homotopical}, \cite{BarneaJoachimMahanta2017ModelStructureProjective}.

The approach taken by Arone--Barnea--Schlank in \cite{AroneBarneaSchlank2021SpectralPresheavesV2} is similar to how one constructs the homotopy theory of (pointed) CW-complexes and its stabilization, but with the role of the $0$-sphere $S^0$ as the basic building block replaced by the finite-dimensional matrix algebras $\{M_n\}_{n \geq 1}$. In particular, they construct an $\infty$-category of \emph{noncommutative CW-complexes} and its stabilization, the $\infty$-category of \emph{noncommutative CW-spectra}. One of their main results is then the following theorem.

\begin{theoremA}[\cite{AroneBarneaSchlank2021SpectralPresheavesV2}]\label{theoremA}
    There exists a symmetric monoidal spectrum-enriched category $\mathcal{M}_s$, whose objects can be thought of as suspension spectra of matrix algebras, such that the category of spectral presheaves on $\mathcal{M}_s$ models the symmetric monoidal $\infty$-category of noncommutative CW-spectra when equipped with the projective model structure and the Day convolution product.
\end{theoremA}

Their proof uses spectrum-enriched $\infty$-categories and an $\infty$-categorical version of a theorem by Schwede--Shipley \cite[Thm.\ 3.9.3.(iii)]{SchwedeShipley2003ModelCatsAreModuleCats}. It was suggested in \cite[Rem.\ 1.4]{AroneBarneaSchlank2021SpectralPresheavesV2} that the techniques of \cite{BlomMoerdijk2020SimplicialProV1} allow for a more direct proof of this theorem, using standard results on stable model categories. The aim of this note is to present such a proof. In particular, we show that the result above can be proved without the use of enriched $\infty$-categories, avoiding the technicalities that such an approach comes with. Furthermore, our proof has the added benefit of producing convenient model categories of noncommutative CW-complexes and noncommutative CW-spectra.

A rough outline of our proof is as follows: We first show that for a suitably chosen subcategory $\bfD^{\mathrm{op}}$ of the category of $C^*$-algebras, such as the category of separable $C^*$-algebras, one can endow $\Ind(\bfD)$ with a model structure whose cofibrant objects are exactly the (retracts of) noncommutative CW-complexes. This model category is then stabilized by considering a certain model structure on the category of functors from $\pfinsSet$ to $\Ind(\bfD)$. This produces a model category enriched over Lydakis's stable model category of simplicial functors \cite{Lydakis1998SimplicialFunctors}, and one can then apply a modified version of \cite[Thm.\ 3.9.3.(iii)]{SchwedeShipley2003ModelCatsAreModuleCats} to conclude \Cref{theoremA}.

The techniques of \cite{BlomMoerdijk2020SimplicialProV1} are used to construct the model category of noncommutative CW-complexes in \Cref{ssec:ModelCategoryNCW}. For the convenience of the reader, we discuss (a simplification of) these techniques in the \hyperref[appendix:minimal]{appendix}.

Throughout this note, we use the convention that all topological spaces are compactly generated weak Hausdorff.

\paragraph{Notation.} The categories considered in this paper often admit several useful enrichments. To avoid confusion, we will generally write $\Map_{\mathbf{V}}(-,-)$ to denote the hom-objects of a $\mathbf{V}$-enriched category, including the base of enrichment in the notation. For brevity, we will denote a simplicial enrichment by $\MapsSet(-,-)$ and an enrichment in pointed simplicial sets by $\MapsSetp(-,-)$.

\paragraph{Acknowledgements.} The author would like to thank Floris Elzinga and Makoto Yamashita for bringing to his attention a counterexample to the claim that the maximal tensor product of separable $C^*$-algebras preserves pullbacks.

\section{Noncommutative CW-complexes}

We start this section with a brief introduction to the category of $C^*$-algebras and its subcategories that we will be interested in. We then show how to construct a model category out of such a subcategory that describes the homotopy theory of noncommutative CW-complexes. Finally, we discuss the symmetric monoidal structure on this model category induced by the minimal and maximal tensor product of $C^*$-algebras.

\subsection{The category of \texorpdfstring{$C^*$}{C*}-algebras}
Let $\CAlg$ denote the category of (not necessarily unital) $C^*$-algebras and $*$-homomorphisms. The Gelfand duality theorem states that the category of pointed compact Hausdorff spaces $\CHaus_*$ is dual to the category of \emph{commutative} $C^*$-algebras $\cCAlg$, with the equivalence in the direction $\CHaus_* \to (\cCAlg)^{\mathrm{op}}$ given by sending a pointed space $(X,x)$ to the algebra $C_0(X)$ of continuous basepoint preserving functions $(X,x) \to (\bbC,0)$.

As described in \cite[Rem.\ 2.5]{Uuye2013Homotopical}, the category $\CAlg$ of $C^*$-algebras is enriched over the category $\Top_*$ of (pointed) compactly generated weak Hausdorff spaces and admits cotensors by pointed compact Hausdorff spaces. Explicitly, for a $C^*$-algebra $B$ and a pointed compact Hausdorff space $(X,x)$, the cotensor $B^X$ in $\CAlg$ is defined as the $C^*$-algebra of continuous basepoint preserving maps $(X,x) \to (B,0)$ endowed with the supremum norm. We view $\CAlg$ as enriched in $\Top$ by forgetting the basepoints of the hom-spaces. The cotensor $B^Y$ of a $C^*$-algebra $B$ by an \emph{unpointed} compact Hausdorff space $Y$ also exists and is given by the $C^*$-algebra of \emph{all} continuous maps $Y \to B$ (cf. \cite[Lem.\ 2.4]{Uuye2013Homotopical}). Equivalently, it is the cotensor of $B$ by the pointed compact Hausdorff space $Y_+$.

Throughout the rest of this paper, we let $\bfD^{\mathrm{op}}$ be a full subcategory of $\CAlg$ satisfying the following properties:\footnote{We write $\bfD^{\mathrm{op}}$ for this subcategory since we will mainly work with its opposite category $\bfD$ below.}
\begin{enumerate}[(D1)]
	\item\label{D1} The category $\bfD^{\mathrm{op}}$ is essentially small.
	\item\label{D2} For every $n \geq 1$, the $C^*$-algebra $M_n$ of $n$-by-$n$ matrices is contained in $\bfD^{\mathrm{op}}$.
	\item\label{D3} For any $A \in \bfD^{\mathrm{op}}$, the cotensor $A^I$ by the unit interval $I$ is also an object of $\bfD^{\mathrm{op}}$.
	\item\label{D4} The full subcategory $\bfD^{\mathrm{op}}$ is closed under finite limits.
\end{enumerate}

Examples of subcategories $\bfD^{\mathrm{op}}$ to keep in mind are those of all separable $C^*$-algebras and those of all separable $C^*$-algebras that are furthermore nuclear. Note that separability of a $C^*$-algebra $A$ implies that there exists a countable subset $Z \subset A$ together with a surjection $Z^\bbN \to A$, hence the cardinality of a separable $C^*$-algebra is at most $2^{\aleph_0}$. This shows that \ref{D1} must hold for these two examples. Furthermore, finite-dimensional $C^*$-algebras are always separable and nuclear, hence \ref{D2} holds as well. We leave it as an exercise to the reader to verify \ref{D3}. Property \ref{D4} follows since the terminal object is clearly nuclear and separable, and both the subcategories of nuclear and of separable $C^*$-algebras are closed under pullbacks by \cite[Rem.\ 3.5]{Pedersen1999PullbackPushout}.

\begin{remark}\label{remark:MinimalD}
	It is worth pointing out that, up to equivalence, there is a minimal choice of a full subcategory $\CAlg$ satisfying \ref{D1}-\ref{D4}. Namely, let $\mathcal{D}$ be the class of all full subcategories of $\CAlg$ satisfying \ref{D1}-\ref{D4} and that are furthermore isomorphism-closed. Then the intersection
	\[\bfD^{\mathrm{op}}_{\mathrm{min}} := \bigcap_{\bfD^{\mathrm{op}} \in \mathcal{D}} \bfD^{\mathrm{op}}\]
	again satisfies \ref{D1}-\ref{D4} and is, up to equivalence, the smallest such full subcategory.
\end{remark}

\subsection{The model category of noncommutative CW-complexes}\label{ssec:ModelCategoryNCW}

We will now construct a model category describing the homotopy theory of noncommutative CW-complexes. Let $\bfD$ denote the opposite category of a full subcategory $\bfD^{\mathrm{op}}$ of $\CAlg$ satisfying \ref{D1}-\ref{D4}. For a $C^*$-algebra $A$ in $\bfD^{\mathrm{op}}$, we will write $\op{A}$ for the corresponding object in $\bfD$ in an attempt to avoid confusion. In particular, since cotensors are formally dual to tensors, we see that $\op{A} \otimes X = \op{A^X}$ for all cotensors $A^X$ that $\bfD^{\mathrm{op}}$ admits. Let $\MAlgu \subset \Ob(\bfD)$ denote the set of matrix algebras; that is, $\MAlgu = \{\op{M_n}\}_{n \geq 1}$.

It is easy to verify that the cotensor functor $(\mhyphen)^I$ preserves finite limits, so $\bfD$ satisfies the properties spelled out in \Cref{appendix:exampleTopological}. In particular, the category $\bfD$ together with the set of objects $\MAlgu$ is a minimal cofibration test category in the sense of \Cref{appendix:definition}, with its simplicial hom-sets defined by $\MapsSet(\op{A},\op{B}) := \SingTop (\MapTop(B,A))$. By combining \Cref{appendix:mainthm} and \Cref{appendix:remarkTopologicalGenCofibs}, we obtain the following result.

\begin{theorem}\label{theorem:TheModelStructureUnstable}
	Let $\bfD$ be as above. There exists a cofibrantly generated simplicial model structure on $\Ind(\bfD)$ such that
	\begin{enumerate}[(i)]
		\item \label{theorem:TheModelStructureUnstable:item1} a map $C \to D$ is a weak equivalence or fibration if and only if for every $n \geq 1$, the map
		\[\MapsSet(\op{M_n},C) \to \MapsSet(\op{M_n},D)\]
		is a weak equivalence or Kan fibration, respectively,
		\item \label{theorem:TheModelStructureUnstable:item2} any object of $\Ind(\bfD)$ is fibrant,
		\item \label{theorem:TheModelStructureUnstable:item3} a set of generating cofibrations is given by
		\[\{\op{M_n} \otimes \partial D^k \to \op{M_n} \otimes D^k \mid n \geq 1,\quad k \geq 0 \} \]
		and a set of generating trivial cofibrations by
		\[\{\op{M_n} \otimes (D^k \times \{0\}) \to \op{M_n} \otimes (D^k \times I) \mid n \geq 1,\quad k \geq 0 \},\]
		\item \label{theorem:TheModelStructureUnstable:item4} the weak equivalences are stable under filtered colimits.
	\end{enumerate}
\end{theorem}

We will call this model category the \emph{model category of (pointed) non-commutative CW-complexes}.

\begin{remark}
	Note that the $C^*$-algebra $\{0\}$ defines both an initial and a terminal object of $\bfD$, and hence also of $\Ind(\bfD)$. In particular, the simplicial enrichment of $\Ind(\bfD)$ can be upgraded to an enrichment in the category of pointed simplicial sets, where the basepoint of $\MapsSetp(C,D)$ is the unique map $C \to D$ that factors through $\{0\}$. This makes $\Ind(\bfD)$ into a $\s\Set_*$-enriched model category. 
\end{remark}

\begin{remark}
The category $\Ind(\bfD)$ clearly depends on the choice of subcategory $\bfD^{\mathrm{op}}$ of $\CAlg$. However, for the model structure this is not really the case; at least, not up to Quillen equivalence. To see this, first note that we may assume without loss of generality that $\bfD^{\mathrm{op}}$ is an isomorphism-closed full subcategory. By \Cref{remark:MinimalD}, we have an inclusion $\bfD^{\mathrm{op}}_\mathrm{min} \hookrightarrow \bfD^{\mathrm{op}}$. A proof similar to that of Proposition 7.8 of \cite{BlomMoerdijk2020SimplicialProV1} then shows that this inclusion induces a Quillen equivalence $\Ind(\bfD_\mathrm{min}) \rightleftarrows \Ind(\bfD)$ when both categories are endowed with the model structure of \Cref{theorem:TheModelStructureUnstable}.
\end{remark}

\subsection{Finite cell complexes}

Recall the definition of a (finite) cell complex in a cofibrantly generated model category from \cite[Def.\ 10.5.8]{Hirschhorn2003Model}. In the model category $\Ind(\bfD)$ from \Cref{theorem:TheModelStructureUnstable}, an object $C$ is a finite cell complex if the map $\varnothing \to C$ is obtained by attaching cells of the form $\op{M_n} \otimes \partial D^k \cofarrow \op{M_n} \otimes D^k$ a finite number of times; these are precisely the finite $\MAlgu$-cell complexes in the terminology of \Cref{appendix:def:FiniteCellComplex}. In particular, they can be viewed as objects in $\bfD$. It is easy to see that these are exactly the objects called \emph{finite pointed noncommutative CW-complexes} in \cite[Def.\ 2.3]{AroneBarneaSchlank2021SpectralPresheavesV2}. Denote the full simplicial subcategory that they span by $\NCWf$.

\begin{remark}
	Since we do not assume that cells may only be attached to lower-dimensional cells, it would perhaps be better to call these objects noncommutative cell complexes and reserve the name noncommutative CW-complex for objects where this extra assumption is made.
\end{remark}

The $\infty$-category $\mathtt{NCW}$ of noncommutative pointed CW-complexes is defined in \cite[\S 2]{AroneBarneaSchlank2021SpectralPresheavesV2} as $\Indinfty(\hcN(\NCWf))$, where $\hcN$ denotes the homotopy coherent nerve and $\Indinfty$ the $\infty$-categorical ind-completion in the sense of \cite[Def.\ 5.3.5.1]{Lurie2009HTT}. In particular, the following is a direct consequence of \Cref{appendix:propUnderlyingInftyCat}.

\begin{proposition}\label{prop:comparison-non-monoidal}
The underlying $\infty$-category of the model category $\Ind(\bfD)$ from \Cref{theorem:TheModelStructureUnstable} is equivalent to the $\infty$-category of noncommutative pointed CW-complexes defined in \cite{AroneBarneaSchlank2021SpectralPresheavesV2}.
\end{proposition}

\subsection{Tensor products}\label{ssec:Tensorproducts}

Both the maximal and the minimal tensor product endow $\CAlg$ with a symmetric monoidal structure. It is natural to ask whether these tensor products extend to tensor products on $\Ind(\bfD)$ and how these interact with the model structure from \Cref{theorem:TheModelStructureUnstable}. The first of these questions is easy to answer: if the (maximal or minimal) tensor product $\otimes$ restricts to a symmetric monoidal structure on $\bfD^{\mathrm{op}}$ (and hence on $\bfD$), then the canonical extension
\begin{equation}\label{eq:MonoidalStructureIndD}
	C \otimes D :=  \{c_i \otimes d_j \}_{(i,j) \in I \times J} \cong \colim_{(i,j) \in I \times J} c_i \otimes d_j,
\end{equation}
defines a symmetric monoidal structure on $\Ind(\bfD)$, where $C = \{c_i\}_{i \in I}$ and $D = \{d_j\}_{j \in J}$ are arbitrary objects of $\Ind(\bfD)$. Furthermore, if the tensor product on $\bfD$ preserves finite colimits in both variables, then its extension to $\Ind(\bfD)$ admits a right adjoint in both variables, meaning that $\Ind(\bfD)$ is closed monoidal. This follows since the filtered colimit preserving extension of a finite colimit preserving functor always admits a right adjoint (cf. \cite[\S 2.2]{BlomMoerdijk2020SimplicialProV1}).

Let us consider the cases where $\bfD^{\mathrm{op}}$ is the category of separable or the category of nuclear separable $C^*$-algebras. In both cases, $\bfD^{\mathrm{op}}$ is closed under the minimal as well as the maximal tensor product; hence, they extend to symmetric monoidal structures on $\Ind(\bfD)$. However, in the case where $\bfD^{\mathrm{op}}$ is the full subcategory of all separable $C^*$-algebras, neither of these tensor products preserves pullbacks, so they do not make $\Ind(\bfD)$ into a closed symmetric monoidal category.\footnote{It is claimed in \cite[Rem.\ 3.10]{Pedersen1999PullbackPushout} that the maximal tensor product preserves pullbacks in both of its variables. However, since being a monomorphism can be expressed through a pullback diagram, this would imply that the maximal tensor product preserves monomorphisms. This is wrong for (separable) $C^*$-algebras in general, as also mentioned in \cite[Rem.\ 3.10]{Pedersen1999PullbackPushout}.} However, in the case of nuclear $C^*$-algebras, the situation is much better: these are by definition the $C^*$-algebras for which the minimal and maximal tensor products agree, and by \cite[Thm.\ 3.9]{Pedersen1999PullbackPushout} the tensor product of nuclear $C^*$-algebras preserves pullbacks (and hence finite limits) in each of its variables. In particular, taking $\bfD^{\mathrm{op}}$ to be the category of nuclear separable $C^*$-algebras, we obtain a closed symmetric monoidal structure on $\Ind(\bfD)$. Moreover, this tensor product turns out to interact well with the model structure.

\begin{proposition}
	If $\bfD^{\mathrm{op}}$ is the category of nuclear separable $C^*$-algebras, then $\Ind(\bfD)$ is a symmetric monoidal model category when equipped with the monoidal structure given in \eqref{eq:MonoidalStructureIndD}.
\end{proposition}

\begin{proof}
	It follows from the above that the tensor product on $\bfD$ preserves finite colimits in each of its variables separately and hence that it extends to a closed symmetric monoidal structure on $\Ind(\bfD)$. To see that it is a monoidal model category, note that the pushout product of $\op{M_n} \otimes \partial D^k \to \op{M_n} \otimes D^k$ and $\op{M_{n'}} \otimes \partial D^{k'} \to \op{M_{n'}} \otimes D^{k'}$ is
	\[\op{M_{n \times n'}} \otimes \left(D^k \times \partial D^{k'} \smashoperator{\bigcup_{\partial D^k \times \partial D^{k'}}} \partial D^k \times D^{k'}\right) \to \op{M_{n \times n'}} \otimes (D^k \times D^{k'})\]
	and that the pushout product of $\op{M_n} \otimes \partial D^k \to \op{M_n} \otimes D^k$ and $\op{M_{n'}} \otimes ( D^{k'} \times \{0\}) \to \op{M_{n'}} \otimes ( D^{k'} \times I)$ is
	\[\op{M_{n \times n'}} \otimes \left(D^k \times D^{k'} \times \{0\} \smashoperator{\bigcup_{\partial D^k \times D^{k'} \times \{0\}}} \partial D^k \times D^{k'} \times I \right) \to \op{M_{n \times n'}} \otimes (D^k \times D^{k'} \times I).\]
	Here we use that $M_n \otimes M_{n'} = M_{n \times n'}$ and that the maximal tensor product of $C^*$-algebras is compatible with cotensors by compact Hausdorff spaces (cf. \cite[Lem.\ 2.2]{Uuye2013Homotopical}). It is clear that these maps are cofibrations and trivial cofibrations in $\Ind(\bfD)$, respectively, so by \cite[Cor.\ 4.2.5]{Hovey1999ModelCats} we conclude that $\Ind(\bfD)$ is a symmetric monoidal model category.
\end{proof}

Since $\Ind(\bfD)$ is a symmetric monoidal model category, its underlying $\infty$-category obtains a closed symmetric monoidal structure by \cite[Prop.\ 4.1.7.10]{Lurie2017HA}.
The $\infty$-category of noncommutative pointed CW-complexes $\mathtt{NCW}$ defined in \cite{AroneBarneaSchlank2021SpectralPresheavesV2} also has a symmetric monoidal structure.
Under the equivalence of \Cref{prop:comparison-non-monoidal}, these agree:

\begin{proposition}\label{prop:comparison-monoidal}
    Let $\bfD^{\mathrm{op}}$ be the category of nuclear separable $C^*$-algebras. Then the equivalence of $\infty$-categories from \Cref{prop:comparison-non-monoidal} can be upgraded to a strong symmetric monoidal equivalence.
\end{proposition}

\begin{proof}
    The symmetric monoidal structure on $\mathtt{NCW} = \Indinfty(\hcN(\NCWf))$ defined in \cite[\S 2]{AroneBarneaSchlank2021SpectralPresheavesV2} is uniquely determined by the facts that it preserves filtered colimits in each variable and that the inclusion $\hcN(\NCWf) \hookrightarrow \Indinfty(\hcN(\NCWf))$ is strong symmetric monoidal.
    The underlying symmetric monoidal $\infty$-category of $\Ind(\bfD)$, constructed as in \cite[Prop.\ 4.1.7.10]{Lurie2017HA}, also has these properties by construction.
\end{proof}

\section{Noncommutative CW-spectra}

In this section, we will study the stabilization of the model category of noncommutative CW-complexes. In order to obtain a stable model category that is enriched in some category of spectra, we will work with a stabilization based on Lydakis' stable model category of simplicial functors. We prove that this stabilization is equivalent to the category of spectral presheaves on a certain category $\scalM$, and then show that this model category is Quillen equivalent to the category of spectral presheaves on the category $\mathcal{M}_s$ defined in \cite{AroneBarneaSchlank2021SpectralPresheavesV2}. In particular, this recovers one of the main results of that paper.

\subsection{Stabilizing the category of noncommutative CW-complexes}

Write $\Sp(\Ind(\bfD))$ for the category of (sequential) spectrum objects in $\Ind(\bfD)$ as defined in \cite[Def.\ 2.1.1]{Schwede1997SpectraCotangent}. By item \ref{theorem:TheModelStructureUnstable:item2} of \Cref{theorem:TheModelStructureUnstable}, the model structure for noncommutative CW-complexes is right proper; hence, the stable model structure on $\Sp(\Ind(\bfD))$ exists by \cite[Thm.\ 5.23]{BiedermannRondigs2014CalculusII}.

\begin{proposition}\label{prop:comparison-stable-non-monoidal}
	The underlying $\infty$-category of the stable model structure on $\Sp(\Ind(\bfD))$ is the stabilization of the underlying $\infty$-category of $\Ind(\bfD)$. In particular, $\Sp(\Ind(\bfD))$ models the $\infty$-category of noncommutative CW-spectra from \cite[\S 6]{AroneBarneaSchlank2021SpectralPresheavesV2}.
\end{proposition}

\begin{proof}
	This is proved analogously to \cite[Prop.\ 4.2.4]{Robalo2014Motivic}.
\end{proof}

Unfortunately, this model category does not come with a spectral enrichment, so Theorem 3.9.3.(iii) of \cite{SchwedeShipley2003ModelCatsAreModuleCats} or the generalization of that theorem given in \cite[Thm.\ 1.36]{MayGuillou2020EnrichedPresheaf} cannot be applied directly. To solve this, we will work with Lydakis' model structure for linear functors.

Recall that a simplicial set is called \emph{finite} if it has finitely many nondegenerate simplices. For any pointed simplicial category $\mathcal E$, let $\SF(\mathcal E)$ denote the category of pointed simplicial functors $\pfinsSet \to \mathcal E$, where $\pfinsSet$ is the category of pointed finite simplicial sets. If $\mathcal{E}$ is equipped with a model structure, then we will call such a pointed simplicial functor a \emph{homotopy functor} if it preserves weak (homotopy) equivalences and \emph{linear} if it furthermore sends homotopy pushouts to homotopy pullbacks. In \cite{Lydakis1998SimplicialFunctors}, Lydakis showed that there exists a left Bousfield localization of the projective model structure on $\sSF := \SF(\s\Set_*)$ in which the fibrant objects are exactly the pointwise fibrant linear homotopy functors (called the \emph{stable model structure of simplicial functors} there). It is shown that this is a symmetric monoidal model structure under the Day convolution product and that it is Quillen equivalent to the model category of sequential spectra \cite[Thm.\ 11.3]{Lydakis1998SimplicialFunctors}. Let us denote this model category by $\sSFlin$.

In \cite{BiedermannRondigs2014CalculusII}, this construction is extended to a more general setting that also applies to $\Ind(\bfD)$. In particular, it follows from Theorem 5.8 of that paper applied to the case $n=1$ that the analogous left Bousfield localization for $\SF(\Ind(\bfD))$ exists, which we will denote by $\SFlin(\Ind(\bfD))$. Theorems 5.24 and 5.27 of \cite{BiedermannRondigs2014CalculusII} then imply that $\SFlin(\Ind(\bfD))$ is equivalent to the stable model structure on $\Sp(\Ind(\bfD))$. Unlike $\Sp(\Ind(\bfD))$, it can be shown that this model structure \emph{does} come with a canonical spectral enrichment.

\begin{proposition}\label{prop:SFlinEnrichment}
	$\SFlin(\Ind(\bfD))$ is an $\sSFlin$-enriched model structure.
\end{proposition}

\begin{proof}
	The enrichment, tensor, and cotensor over $\sSF$ are defined by formulas analogous to those of the Day convolution: Given a simplicial set $K$, write $K_! \colon \pfinsSet \to \pfinsSet$ for the functor defined by $K_!(M) = M \wedge K$. The tensor is defined as the coend
	\[(X \otimes F)(K) = \qquad \smashoperator{\int\limits^{(K_1,K_2) \in \pfinsSet \times \pfinsSet}} \quad X(K_1) \otimes F(K_2) \otimes \Map_*(K_1 \wedge K_2, K), \]
	while the enrichment and cotensor are defined by
	\[\MapSF(X,Y)(K) = \MapsSetp(X, Y \circ K_!) \quad \text{and} \quad (X^F)(K) = (X \circ K_!)^F. \]
	We leave it to the reader to verify that these indeed constitute an enrichment of $\SF(\Ind(\bfD))$ over $\sSF$ that is both tensored and cotensored.
	
	It can be shown that this makes the projective model structure on $\SF(\Ind(\bfD))$ into an $\sSF$-enriched model category (with respect to the projective model structure on $\sSF$) by studying pushout products of generating (trivial) cofibrations. This is similar to the proof of \cite[Thm.\ 12.3]{Lydakis1998SimplicialFunctors} and is left to the reader. To see that $\SFlin(\Ind(\bfD))$ is furthermore an $\sSFlin$-enriched model category, it thus suffices to show that for any pair of cofibrations $X \cofarrow Y$ in $\SFlin(\Ind(\bfD))$ and $F \cofarrow G$ in $\sSFlin$, one of which is trivial, the pushout product
	\begin{equation}\label{eq:PushoutProductDay}
		Y \otimes F \cup_{X \otimes F} X \otimes G \to Y \otimes G
	\end{equation}
	is a trivial cofibration in $\SFlin(\Ind(\bfD))$. We treat the case where $F \cofarrow G$ is trivial; the other case is similar. Without loss of generality, assume that $X \to Y$ is a generating cofibration, so in particular that $X$ and $Y$ are cofibrant. Since we already know that \eqref{eq:PushoutProductDay} is a cofibration, it suffices to show that this map has the left lifting property with respect to fibrations between fibrant objects $B \fibarrow A$ (cf.\ \cite[Lem.\ 8.43]{HeutsMoerdijk2020Trees}). By adjunction, this is equivalent to proving that
	\[\MapSF(Y,B) \to \MapSF(Y,A) \times_{\MapSF(X,A)} \MapSF(X,B) \]
	is a fibration in $\sSFlin$. We already know that this is a fibration in the projective model structure, so by \cite[Prop.\ 3.3.16]{Hirschhorn2003Model} it suffices to show that its domain and codomain are fibrant in $\sSFlin$. Since $X$ and $Y$ are cofibrant and $A$ and $B$ are fibrant, this follows if we can show that for any cofibrant object $Z$ and fibrant object $L$ in $\SFlin(\Ind(\bfD))$, the hom-object $\MapSF(Z,L)$ is fibrant in $\sSFlin$. It is projectively fibrant since the projective model structure is $\sSF$-enriched. It is a homotopy functor since for any weak homotopy equivalence $K \wearrow M$ in $\pfinsSet$, the natural map $L \circ K_! \to L \circ M_!$ is a pointwise equivalence between pointwise fibrant functors, hence $\MapsSetp(Z, L \circ K_!) \to \MapsSetp(Z, L \circ M_!)$ is a weak equivalence. Finally, to see that $\MapSF(G,L)$ is linear, it suffices to show that
	\[\MapSF(G,L) \to \Omega \circ \MapSF(G,L) \circ \Sigma \]
	is a pointwise weak equivalence. This follows since the right-hand side is isomorphic to $\MapSF(G,\Omega L \Sigma)$ and, since $L$ is linear, the map $L \to \Omega L \Sigma$ is a pointwise equivalence between projectively fibrant functors.
\end{proof}

Recall from \Cref{ssec:Tensorproducts} that if we take $\bfD^{\mathrm{op}}$ to be the category of nuclear separable $C^*$-algebras, then $\Ind(\bfD)$ is a closed symmetric monoidal model category. In particular, the Day convolution product endows $\Fun(\pfinsSet, \Ind(\bfD))$ with a closed symmetric monoidal structure, which turns out to be compatible with the model structure.

\begin{proposition}\label{prop:NuclearSymmetricMonoidalSpectra}
	If $\bfD^{\mathrm{op}}$ is the category of nuclear separable $C^*$-algebras, then $\SFlin(\Ind(\bfD))$ is a symmetric monoidal model category under the Day convolution product.
\end{proposition}

\begin{proof}
	This is almost identical to the proof of \Cref{prop:SFlinEnrichment} and is left to the reader.
\end{proof}

By Proposition 4.1.7.10 of \cite{Lurie2017HA}, we obtain a closed symmetric monoidal structure on the underlying $\infty$-category of $\SFlin(\Ind(\bfD))$.
By Theorems 5.24 and 5.27 of \cite{BiedermannRondigs2014CalculusII} and \cref{prop:comparison-stable-non-monoidal}, the underlying $\infty$-category of $\SFlin(\Ind(\bfD))$ is equivalent to the $\infty$-category $\mathtt{NSp}$ of noncommutative CW-spectra defined in \cite[\S 6]{AroneBarneaSchlank2021SpectralPresheavesV2}, which also comes with a symmetric monoidal structure.

\begin{proposition}\label{prop:comparison-stable-monoidal}
    The equivalence between the underlying $\infty$-category of $\SFlin(\Ind(\bfD))$ and the $\infty$-category $\mathtt{NSp}$ of noncommutative CW-spectra is symmetric monoidal.
\end{proposition}

\begin{proof}
    The symmetric monoidal structure constructed on $\mathtt{NSp}$ in \cite[\S 6]{AroneBarneaSchlank2021SpectralPresheavesV2} is uniquely determined by the fact that it preserves colimits in each variable and that the functor $\Sigma^\infty \colon \mathtt{NCW} \to \mathtt{NSp}$ is strong symmetric monoidal. The underlying symmetric monoidal $\infty$-category of $\SFlin(\Ind(\bfD))$ also satisfies this.
\end{proof}

Let $S_0$ denote the (pointed) simplicial set consisting of two points. The functor $\SFlin(\Ind(\bfD)) \to \Ind(\bfD)$ that evaluates $X \in \SFlin(\Ind(\bfD))$ at $S^0 \in \pfinsSet$ will be denoted $\Omega^\infty$. It has a left adjoint $\Sigma^\infty$ defined by
\[\Sigma^\infty D \colon \pfinsSet \to \Ind(\bfD); \quad K \mapsto D \otimes K. \]
We will call $\Sigma^\infty D$ the \emph{suspension spectrum} of $D$. It is straightforward to see that this is a Quillen pair (with respect to both the projective and the linear model structure).

We conclude this section by identifying the weak equivalences between fibrant objects in $\SFlin(\Ind(\bfD))$.

\begin{proposition}\label{proposition:DetectingWeakEquivalences}
	Let $f \colon X \to Y$ be a map between fibrant objects in $\SFlin(\Ind(\bfD))$. Then $f$ is a weak equivalence if and only if for every $n \geq 1$, the map
	\begin{equation}\label{eq:MapSFDetectingWEs}
		\MapSF(\Sigma^\infty \op{M_n}, X) \to \MapSF(\Sigma^\infty \op{M_n}, Y)
	\end{equation}
	is a weak equivalence in $\sSFlin$.
\end{proposition}

\begin{proof}
	Note that for every $n \geq 1$, the suspension spectrum $\Sigma^\infty \op{M_n}$ is cofibrant since $\Sigma^\infty$ is left Quillen. The ``only if'' direction follows from the fact that $\SFlin(\Ind(\bfD))$ is $\sSFlin$-enriched. For the other direction, note that the map \eqref{eq:MapSFDetectingWEs} is an equivalence between fibrant objects in $\sSFlin$, hence a pointwise equivalence. Combining this with the definition of $\MapSF$ given in the proof of \Cref{prop:SFlinEnrichment} and the left-adjointness of $\Sigma^\infty$, this shows that
	\[ \MapsSetp(\op{M_n}, X(K)) \to \MapsSetp(\op{M_n}, Y(K))\]
	is a weak equivalence for every $n \geq 1$. By the definition of the weak equivalences in $\Ind(\bfD)$, it follows that $X \to Y$ is a pointwise equivalence, hence an equivalence in $\SFlin(\Ind(\bfD))$.
\end{proof}

\subsection{Noncommutative CW-spectra as spectral presheaves}

We will now identify the model category $\SFlin(\Ind(\bfD))$ of noncommutative CW-spectra with a spectral presheaf category.

\begin{lemma}
	The suspension spectra of matrix algebras $\{\Sigma^\infty \op{M_n}\}_{n \geq 1}$ form a compact generating set in $\SFlin(\Ind(\bfD))$.
\end{lemma}

\begin{proof}
	Compactness follows since $\Sigma^\infty$ is left adjoint and the objects $\op{M_n}$ are compact in $\Ind(\bfD)$, while it follows from \Cref{proposition:DetectingWeakEquivalences} that $\{\Sigma^\infty \op{M_n}\}$ is a generating set.
\end{proof}

Let $\scalM$ denote the full $\sSF$-enriched subcategory of $\SFlin(\Ind(\bfD))$ spanned by the suspension spectra of matrix algebras $\{\Sigma^\infty \op{M_n}\}_{n \geq 1}$. The definition of the $\sSF$-enrichment together with the left adjointness of $\Sigma^\infty$ shows that the hom-objects of $\scalM$ are given by
\begin{equation}\label{eq:scalMHomObjects}
	\MapSF(\Sigma^\infty \op{M_n}, \Sigma^\infty \op{M_k})(K) = \MapsSetp(\op{M_n}, \op{M_k} \otimes K).
\end{equation}
In particular, $\scalM$ can be viewed as a simplicial version of the topological category $\mathcal{M}_s$ defined in \cite[Def.\ 6.3]{AroneBarneaSchlank2021SpectralPresheavesV2}.

Now note that the restricted enriched Yoneda embedding
\begin{align*}
	\mathbb{U} \colon &\SF(\Ind(\bfD)) \to \Fun(\scalM^{\mathrm{op}}, \sSFlin); \\
	&\bbU(X)(\Sigma^\infty \op{M_n}) = \MapSF(\Sigma^\infty \op{M_n}, X)
\end{align*}
admits a left adjoint $\mathbb{T}$ (cf. \cite[Prop.\ 1.10]{MayGuillou2020EnrichedPresheaf}). As mentioned above, $\Sigma^\infty$ is left Quillen, hence the objects $\Sigma^\infty \op{M_n}$ are all cofibrant in $\SFlin(\Ind(\bfD))$. In particular, the restricted Yoneda embedding $\mathbb U$ sends (trivial) fibrations to pointwise (trivial) fibrations, hence it is right Quillen when $\Fun(\scalM^{\mathrm{op}},\sSFlin)$ is endowed with the projective model structure (which exists by Theorem 7.2 of \cite{SchwedeShipley2003Equivalences}). Theorem 1.36 of \cite{MayGuillou2020EnrichedPresheaf} gives conditions under which this adjunction is a Quillen equivalence. One of the conditions translates to every object of $\scalM$ being fibrant in $\SFlin(\Ind(\bfD))$. While this does not hold, their proof still goes through in our particular case.

\begin{theorem}\label{theorem:NoncommutativeSpectraAsSimplicialPresheafCat}
	The adjunction $\bbT \dashv \bbU$ is a Quillen equivalence between the model categories $\Fun(\scalM^{\mathrm{op}}, \sSFlin)$ and $\SFlin(\Ind(\bfD))$. In particular, the $\infty$-category of noncommutative CW-spectra is modeled by $\Fun(\scalM^{\mathrm{op}}, \sSFlin)$.
\end{theorem}

\begin{proof}
	We leave it to the reader to verify that except for the fibrancy condition on the objects of $\scalM$, all conditions of Theorem 1.36 of \cite{MayGuillou2020EnrichedPresheaf} are satisfied.
	
	We claim that the fibrancy condition is not needed in our particular case. A careful inspection of the proof of \cite[Thm.\ 1.36]{MayGuillou2020EnrichedPresheaf} shows that the only reason why the fibrancy of the objects of $\scalM$ is necessary is so that the value of $\bbU$ agrees with that of its right-derived functor $\bbR \bbU$ when applied to an object of $\scalM$ (that is, the map $\bbU (\Sigma^\infty \op{M_n}) \to \bbR \bbU (\Sigma^\infty \op{M_n})$ is a weak equivalence for every $n \geq 1$). Even though the objects $\Sigma^\infty \op{M_n}$ are not fibrant in $\SFlin(\Ind(\bfD))$, by \Cref{lemma:bbUAlreadyDerived} below this property is still satisfied. We conclude that $\bbT \dashv \bbU$ is a Quillen equivalence.
\end{proof}

\begin{lemma}\label{lemma:bbUAlreadyDerived}
	The map $\bbU (\Sigma^\infty \op{M_n}) \to \bbR \bbU( \Sigma^\infty \op{M_n})$ is a weak equivalence in $\Fun(\scalM^{\mathrm{op}}, \sSFlin)$ for every $n \geq 1$.
\end{lemma}

\begin{proof}
	The suspension spectrum $\Sigma^\infty \op{M_n} \colon \s\Set_* \to \Ind(\bfD)$ is a pointwise fibrant homotopy functor, hence by how the model structure on $\SFlin(\Ind(\bfD))$ is constructed in \cite[Thm.\ 5.8]{BiedermannRondigs2014CalculusII}, an explicit fibrant replacement can be given by
		\begin{equation*}
			\Sigma^\infty \op{M_n} \wearrow \colim_{m \in \bbN} \Omega^m \circ \Sigma^\infty \op{M_n} \circ \Sigma^m.
		\end{equation*}
	It therefore suffices to show that the map
	\begin{equation}\label{eq:bbUAppliedToStabilization}
		\bbU(\Sigma^\infty \op{M_n})(\Sigma^\infty \op{M_k}) \to \bbU(\colim_{m \in \bbN} \Omega^m \Sigma^\infty \op{M_n} \circ \Sigma^m)(\Sigma^\infty \op{M_k})
	\end{equation}
	is an equivalence in $\sSFlin$ for every $k,n \geq 1$. Because of the compactness of the object $\Sigma^\infty \op{M_k}$, the right-hand side is isomorphic to
	\begin{align*} 
		&\bbU(\colim_{m \in \bbN} \Omega^m \Sigma^\infty \op{M_n} \circ \Sigma^m)(\Sigma^\infty \op{M_k}) = \MapSF(\Sigma^\infty \op{M_k}, \colim_{m \in \bbN} \Omega^m \Sigma^\infty \op{M_n} \circ \Sigma^m) \cong \\
		&\colim_{m \in \bbN} \Omega^m \MapSF(\Sigma^\infty \op{M_k}, \Sigma^\infty \op{M_n}) \circ \Sigma^m = \colim_{m \in \bbN} \Omega^m \bbU(\Sigma^\infty \op{M_n})(\Sigma^\infty \op{M_k})  \circ \Sigma^m.
	\end{align*}
	In particular, the map \eqref{eq:bbUAppliedToStabilization} is simply the stabilization map
	\[\bbU(\Sigma^\infty \op{M_n})(\Sigma^\infty \op{M_k}) \to \colim_{m \in \bbN} \Omega^m \bbU(\Sigma^\infty \op{M_n})(\Sigma^\infty \op{M_k}) \circ \Sigma^m,\]
	which is an equivalence in $\sSFlin$ since $\bbU(\Sigma^\infty \op{M_n})(\Sigma^\infty \op{M_k})$ is a pointwise fibrant homotopy functor.
\end{proof}

The symmetric monoidal structure of $\SFlin(\Ind(\bfD))$ from \Cref{prop:NuclearSymmetricMonoidalSpectra} restricts to a symmetric monoidal structure on $\scalM$. It is not hard to see that the projective model structure on $\Fun(\scalM^{\mathrm{op}},\sSFlin)$ is symmetric monoidal under the Day convolution. In the case that $\bfD^{\mathrm{op}}$ is the category of nuclear separable $C^*$-algebras, the Quillen equivalence from above can be upgraded to a monoidal one.

\begin{theorem}
	In the case that $\bfD$ is the category of separable nuclear $C^*$-algebras, the Quillen equivalence of \Cref{theorem:NoncommutativeSpectraAsSimplicialPresheafCat} is strong monoidal.
\end{theorem}

\begin{proof}
	We need to show that $\bbT$ is strong monoidal. By the universal property of Day convolution, it suffices to show that $\bbT$ is strong monoidal when restricted to the image of the enriched Yoneda embedding $\scalM \hookrightarrow \Fun(\scalM^{\mathrm{op}}, \sSFlin)$. This holds since the composition of $\bbT$ with the enriched Yoneda embedding is simply the inclusion $\scalM \hookrightarrow \SFlin(\Ind(\bfD))$, which is strong monoidal by construction.
\end{proof}

\subsection{Simplicial versus topological functors}

Recall that spectra can also be modeled as pointed linear topological functors from the category of finite pointed CW-complexes $\CWf$ to the category of all pointed (compactly generated weak Hausdorff) spaces $\Top_*$, called $\mathscr{W}$-spectra in \cite{MMSS2001DiagramSpectra}. We will denote the model category of $\mathscr{W}$-spectra by $\SpW$. This is a symmetric monoidal model category under the Day convolution product.

Since the category of $C^*$-algebras naturally comes with a $\Top_*$-enrichment, it is perhaps more natural to consider the $\SpW$-enriched analogue of $\scalM$. This is the category denoted by $\mathcal{M}_s$ in \cite{AroneBarneaSchlank2021SpectralPresheavesV2}, but we will denote it by $\tcalM$ to emphasize that it is the "topological" analogue of $\scalM$. Its objects are the positive natural numbers $n$ (which we think of as the suspension spectra of the matrix algebras $M_n$) and for two objects $n$ and $k$, the hom-object $\MapSpW(n, k)$ is defined by
\[\CWf \to \Top_*; \quad W \mapsto \MapTopp(\op{M_n}, \op{M_k} \otimes W). \]
Composition is defined as in \cite[Def.\ 5.12]{AroneBarneaSchlank2021SpectralPresheavesV2}. The tensor product of matrix algebras endows this category with a symmetric monoidal structure.

It is shown in \cite[Thm.\ 19.11]{MMSS2001DiagramSpectra} that there is a (strong monoidal) Quillen equivalence $\bbPT : \sSFlin \rightleftarrows \SpW : \bbSU$, where the right adjoint is defined by
\[\bbSU(F) \colon \pfinsSet \to \s\Set; \quad \bbSU(F)(K) = \SingTop(F(|K|)).\]
It follows from the characterization of the hom-sets given in \eqref{eq:scalMHomObjects} that $\scalM$ is isomorphic to the category $\bbSU(\tcalM)$ obtained by applying $\bbSU$ to each hom-object of $\tcalM$, and one easily sees that their symmetric monoidal structures agree. In particular, by applying $\bbSU$ to the values of a functor $F \in \Fun(\tcalM^{\mathrm{op}}, \SpW)$, we obtain a functor
\[\bbSU \colon \Fun(\tcalM^{\mathrm{op}}, \SpW) \to \Fun(\scalM^{\mathrm{op}}, \sSFlin).\]

\begin{theorem}
	The functor $\bbSU \colon \Fun(\tcalM^{\mathrm{op}}, \SpW) \to \Fun(\scalM^{\mathrm{op}}, \sSFlin)$ is the right adjoint of a strong monoidal Quillen equivalence, where both the domain and codomain are endowed with the projective model structure. In particular, $\Fun(\tcalM^{\mathrm{op}}, \SpW)$ models the $\infty$-category of non-commutative CW-spectra.
\end{theorem}

\begin{proof}
	The fact that $\bbSU$ is the right adjoint of a Quillen equivalence follows by combining Theorems 6.5.(2) and 7.2 of \cite{SchwedeShipley2003Equivalences}. To see that it is a strong monoidal Quillen equivalence, we need to show that the left adjoint $\bbPT_!$ is strong monoidal. By the universal property of Day convolution, it suffices to show this for the restriction of the left adjoint $\bbPT_!$ to the image of the Yoneda embedding $\scalM \hookrightarrow \Fun(\scalM^{\mathrm{op}}, \sSFlin)$. The left adjointness of $\bbPT_!$ implies that it takes representables to representables, hence this reduces to showing that $\scalM \to \bbSU(\tcalM)$ is symmetric monoidal. But this is just the symmetric monoidal isomorphism mentioned above.
\end{proof}

\appendix

\section{Minimal cofibration test categories}\label{appendix:minimal}

The goal of this appendix is to describe a simplification of the main construction from \cite{BlomMoerdijk2020SimplicialProV1}. For two objects $c$ and $d$ in a simplicial category $\bfD$, the simplicial hom-set is denoted by $\MapsSet(c,d)$. Recall that the tensor of an object $c$ in $\bfD$ with a simplicial set $M$ is defined (if it exists) as the essentially unique object $c \otimes M$ such that there are natural\footnote{Here natural means natural in the enriched sense; cf. \cite[\S 1.2]{Kelly1982BasicConcepts}.} isomorphisms $\MapsSet(M, \MapsSet(c,d)) \cong \MapsSet(c \otimes M, d)$ for every $d$ in $\bfD$.

\begin{definition}\label{appendix:definition}
A \emph{minimal cofibration test category} $(\bfD,\bfT)$ consists of an essentially small simplicial category $\bfD$ together with a subset $\bfT \subset \Ob(\bfD)$ of \emph{test objects}, satisfying the following properties:
\begin{enumerate}[(i)]
    \item\label{appendix:definition:item1} The category $\bfD$ admits all finite colimits.
    \item\label{appendix:definition:item2} The category $\bfD$ admits tensors by finite simplicial sets, which moreover commute with finite colimits of $\bfD$.
    \item\label{appendix:definition:item3} For any object $c$ of $\bfD$ and any test object $t \in \bfT$, the simplicial set $\MapsSet(t,c)$ is a Kan complex.
\end{enumerate}
\end{definition}

\begin{example}\label{appendix:exampleTopological}
Let $\bfD$ be an (essentially) small topological category that admits all finite colimits and tensors by the unit interval $I$, and moreover assume that the functor ${- \otimes I} \colon \bfD \to \bfD$ preserves finite colimits. Applying the singular complex functor to the hom-sets of $\bfD$, we obtain a simplicial category that we will also denote by $\bfD$. Then for any set of objects $\bfT \subset \Ob(\bfD)$, the pair $(\bfD, \bfT)$ is a minimal cofibration test category. Items \ref{appendix:definition:item1} and \ref{appendix:mainthm:item3} are obvious. To see that item \ref{appendix:definition:item2} holds, note that tensors by a simplicial set $M$ in the simplicial category $\bfD$ agree with tensors by the geometric realization $|M|$ when $\bfD$ is viewed as a topological category. Since any finite simplicial set is a finite colimit of representables $\Delta[n]$, it suffices to show that $\bfD$, viewed as a topological category, admits tensors by the topological spaces $|\Delta[n]|$. This follows since
\[c \otimes |\Delta[n]| \cong c \otimes (\underbrace{I \times \ldots \times I}_{n\text{-times}}) \cong c \otimes \underbrace{I \otimes \ldots \otimes I}_{n\text{-times}}.\]
\end{example}

For any simplicial category $\bfD$, the ind-category $\Ind(\bfD)$ is again simplicial, with the enrichment defined by
\[\MapsSet(\{c_i\}, \{d_j\}) = \lim_i \colim_j \MapsSet(c_i, d_j).\]

The main result of this appendix is the following.

\begin{theorem}\label{appendix:mainthm}
Let $(\bfD,\bfT)$ be a minimal cofibration test category. Then there exists a cofibrantly generated simplicial model structure on $\Ind(\bfD)$ with the following properties:
\begin{enumerate}[(i)]
    \item\label{appendix:mainthm:item1} A map $C \to D$ is a weak equivalence or fibration if and only if for any $t \in \bfT$, the map
    \[\MapsSet(t,C) \to \MapsSet(t,D) \]
    is a weak homotopy equivalence or Kan fibration, respectively.
    \item\label{appendix:mainthm:item2} Any object of $\Ind(\bfD)$ is fibrant.
    \item\label{appendix:mainthm:item3} A set of generating cofibrations is given by
    \[\mathcal{I} = \{t \otimes \partial \Delta[n] \to t \otimes \Delta[n] \mid n \geq 0,\quad t \in \bfT \} \]
    and a set of generating trivial cofibrations by
    \[\mathcal{J} = \{t \otimes \Lambda^k[n] \to t \otimes \Delta[n] \mid 0 \leq k \leq n,\quad n \neq 0,\quad t \in \bfT \}\]
    \item \label{appendix:mainthm:item4} The weak equivalences are stable under filtered colimits.
\end{enumerate}
\end{theorem}

In light of this theorem, we will call a map in $\Ind(\bfD)$ a \emph{weak equivalence} if $\MapsSet(t,C) \to \MapsSet(t,D)$ is a weak equivalence for any $t \in \bfT$.

\begin{remark}\label{appendix:remarkTopologicalGenCofibs}
Note that the maps $|\partial \Delta[n]| \hookrightarrow |\Delta[n]|$ and $|\Lambda^k[n]| \hookrightarrow |\Delta[n]|$ can be identified with the inclusions $\partial D^n \hookrightarrow D^n$ and $D^{n-1} \times \{0\} \hookrightarrow D^{n-1} \times I$, respectively. In particular, if the minimal cofibration test category $(\bfD,\bfT)$ comes from a topological category as in \Cref{appendix:exampleTopological}, then one can also take
\[\mathcal{I} = \{t \otimes \partial D^n \to t \otimes D^n \mid n \geq 0,\quad t \in \bfT \} \]
and
\[\mathcal{J} = \{t \otimes (D^n \times \{0\}) \to t \otimes (D^n \times I) \mid n \geq 0,\quad t \in \bfT \}\]
as sets of generating (trivial) fibrations, where $D^n$ denotes the $n$-dimensional unit disc.
\end{remark}

\begin{remark}
The somewhat unfortunate name ``minimal cofibration test category'' comes from the fact that if $(\bfD,\bfT)$ is a minimal cofibration test category, then one can form a cofibration test category in the sense of Definition 3.3 of \cite{BlomMoerdijk2020SimplicialProV1} by defining the category of test objects to be the full subcategory $\bfT'$ consisting of objects of the form $t \otimes N$ for any finite simplicial set $N$ and taking as the cofibrations the maps of the form $t \otimes N \cofarrow t \otimes M$ with $N \cofarrow M$ a monomorphism of finite simplicial sets. The trivial cofibrations are then defined as the cofibrations that are also weak equivalences. One can verify that the model structure of \Cref{appendix:mainthm} agrees with the model structure that one obtains from this cofibration test category by applying Theorem 3.9 of \cite{BlomMoerdijk2020SimplicialProV1}. This cofibration test category $(\bfD,\bfT')$ is the smallest structure of a cofibration test category that one can put on $\bfD$ that has the property that $\bfT'$ contains the set $\bfT$.
\end{remark}

The proof relies on the following three lemmas. We call a map $i \colon C \to D$ in a simplicial category that admits tensors by $\Delta[1]$ an \emph{inclusion of a deformation retract} if there exist maps $r \colon D \to C$ and $H \colon D \otimes \Delta[1] \to D$ such that
\begin{enumerate}[(DR1),left=10pt]
    \item $ri = \id_C$
    \item $H$ is a homotopy from $\id_D$ to $ir$ that is constant on $C$.
\end{enumerate}

\begin{lemma} \label{appendix:lemma1}
Inclusions of deformation retracts are stable under pushouts.
\end{lemma}

\begin{proof}
This is analogous to the proof of \cite[Prop.\ 2.4.9]{Hovey1999ModelCats}.
\end{proof}

\begin{lemma} \label{appendix:lemma2}
Any map in the set $\mathcal{J}$ of item \ref{appendix:mainthm:item3} of \Cref{appendix:mainthm} is an inclusion of a deformation retract.
\end{lemma}

\begin{proof}
Let the map $i \colon t \otimes \Lambda^k[n] \to t \otimes \Delta[n]$ in $\mathcal{J}$ be given. By adjunction, constructing a retract $r$ is equivalent to solving a lifting problem of the form
\[\begin{tikzcd}
\Lambda^k[n] \ar[r] \ar[d,tail,"\sim" rot90] & \MapsSet(t,t \otimes \Lambda^k[n]) \\
\Delta[n] \ar[ur,dashed] &
\end{tikzcd}\]
and constructing the desired homotopy $H$ from $\id_D$ to $ir$ is equivalent to solving a lifting problem of the form
\[\begin{tikzcd}
\Lambda^k[n] \times \Delta[1] \cup_{\Lambda^k[n] \times \partial \Delta[1]} \Delta[n] \times \partial \Delta[1] \ar[r] \ar[d,tail,"\sim" rot90] & \MapsSet(t,t \otimes \Delta[n]) \\
\Delta[n] \times \Delta[1] \ar[ur,dashed] &
\end{tikzcd}\]
This is possible by item \ref{appendix:definition:item3} of \Cref{appendix:definition}.
\end{proof}

\begin{lemma} \label{appendix:lemma3}
Any inclusion of a deformation retract is a weak equivalence in $\Ind(\bfD)$.
\end{lemma}

\begin{proof}
This follows since $\MapsSet(t,-) \colon \Ind(\bfD) \to \s\Set$ is a simplicial functor, hence it preserves deformation retracts.
\end{proof}

\begin{proof}[Proof of \Cref{appendix:mainthm}]
The fact that $\Ind(\bfD)$ is a simplicial category that is complete, cocomplete and that admits all (co)tensors is explained in \cite[\S 2.2]{BlomMoerdijk2020SimplicialProV1}. Item \ref{appendix:mainthm:item4} follows since $\MapsSet(t,-)$ preserves filtered colimits for every $t \in \bfT$ and since weak equivalences in $\s\Set$ are stable under filtered colimits.

We now prove that the model structure exists by checking all items of Theorem 11.3.1 of \cite{Hirschhorn2003Model}. It is clear that the class of weak equivalences satisfies the two out of three property and is closed under retracts. The sets $\mathcal{I}$ and $\mathcal{J}$ permit the small object argument since both sets consist of maps between objects that are compact in $\Ind(\bfD)$. It is furthermore clear that any map in $\mathcal{J}$ lies in the saturation of $\mathcal{I}$, since the saturation of $\mathcal{I}$ includes all maps of the form $t \otimes M \cofarrow t \otimes N$ where $M \cofarrow N$ is a monomorphism of simplicial sets and $t \in \bfT$.

It follows from \Cref{appendix:lemma1,appendix:lemma2,appendix:lemma3} that any pushout of a map in $\mathcal{J}$ is a weak equivalence, and transfinite compositions of such maps are again weak equivalences since weak equivalences are stable under cofiltered limits. In particular, any map in the saturation of $\mathcal{J}$ is a weak equivalence.

It follows by adjunction that a map $C \to D$ has the right lifting property with respect to the maps in $\mathcal{J}$ if and only if $\MapsSet(t,C) \to \MapsSet(t,D)$ is a Kan fibration, while it has the right lifting property with respect to the maps in $\mathcal{I}$ if and only if $\MapsSet(t,C) \to \MapsSet(t,D)$ is a trivial Kan fibration. In particular, a map $C \to D$ has the right lifting property with respect to the maps in $\mathcal{I}$ if and only if it is a weak equivalence and has the right lifting property with respect to the maps in $\mathcal{J}$.

To see that every object is fibrant, let $C = \{c_i\} \in \Ind(\bfD)$ be given. For any $t \in \bfT$, one has $\MapsSet(t,C) = \colim_i \MapsSet(t,c_i)$. Since cofiltered limits of Kan complexes are again Kan complexes, we conclude that for any $t \in \bfT$, the simplicial set $\MapsSet(t,C)$ is a Kan complex. In particular, $C$ is fibrant.

Finally, to see that this model structure is simplicial, it suffices to show that for any $t \in \bfT$, any cofibration $M \cofarrow N$ between finite simplicial sets, and any fibration $C \fibarrow D$ in $\Ind(\bfD)$, the map
\[\MapsSet(t \otimes N,C) \to \MapsSet(t \otimes N, D) \times_{\MapsSet(t \otimes M, D)} \MapsSet(t \otimes M,C) \]
is a Kan fibration, which is trivial if $M \cofarrow N$ or $C \fibarrow D$ is. This follows immediately by noting that this map agrees with the pullback-power of $M \cofarrow N$ and $\MapsSet(t,C) \fibarrow \MapsSet(t,D)$.
\end{proof}

\begin{remark}
	In \cite{BlomMoerdijk2020SimplicialProV1}, cofibration test categories with respect to the Joyal model structure on $\s\Set$ were also considered. One can similarly define a \emph{minimal} cofibration test category with respect to the Joyal model structure on $\s\Set$ by weakening item \ref{appendix:definition:item3} of \Cref{appendix:definition} to only require that $\Map(t,c)$ is a quasicategory. One can then prove a theorem analogous to \Cref{appendix:mainthm}, but the proof is slightly more complicated. We leave this as an exercise for the interested reader.
\end{remark}

Finally, it will be useful to have a description of the underlying $\infty$-category of $\Ind(\bfD)$ for a minimal cofibration test category $(\bfD,\bfT)$. We will need the following definition for this.

\begin{definition} \label{appendix:def:FiniteCellComplex}
	Let $(\bfD,\bfT)$ be a minimal cofibration test category. A finite $\bfT$-cell complex is an object $d \in \bfD$ such that the map $\varnothing \to d$ is a finite composition of pushouts of maps of the form $t \otimes \partial \Delta[n] \to t \otimes \Delta[n]$. In other words, one can write $\varnothing \to d$ as a finite composition 
	\[\varnothing = d_0 \to d_1 \to \ldots \to d_m = d\]
	such that for every $0 \leq i < m$, there exists a pushout square of the form
	\[\begin{tikzcd}
		t \otimes \partial \Delta[n] \ar[r] \ar[d] & d_i \ar[d] \\
		t \otimes \Delta[n] \ar[r] & d_{i+1} \ar[ul, phantom, very near start, "\ulcorner"]
	\end{tikzcd} \]
	for some $t \in \bfT$ and $n \geq 0$.
\end{definition}

Denote the full simplicial subcategory of $\bfD$ spanned by the finite $\bfT$-cell complexes by $\mathrm{cell}_\mathrm{fin}(\bfT)$. A proof similar to that of Theorem A.2 of \cite{BlomMoerdijk2020SimplicialProV1} then shows the following.

\begin{proposition}\label{appendix:propUnderlyingInftyCat}
	Let $(\bfD, \bfT)$ be a minimal cofibration test category. Then the underlying $\infty$-category of $\Ind(\bfD)$ is equivalent to $\Indinfty(\hcN(\mathrm{cell}_\mathrm{fin}(\bfT)))$, where $\hcN$ denotes the homotopy coherent nerve of a simplicial category and $\Indinfty$ the ind-completion of an $\infty$-category in the sense of \cite[Def.\ 5.3.5.1]{Lurie2009HTT}.
\end{proposition}

\renewcommand*{\bibfont}{\normalfont\footnotesize}
\printbibliography
\end{document}